\documentclass[11pt, reqno]{amsart}
\usepackage{amssymb, amsthm, amsmath, amsfonts,mathtools}
\usepackage{array, epsfig}
\usepackage{bbm}

\usepackage{hyperref}
\usepackage[numbers,square]{natbib}
\usepackage{color}
\allowdisplaybreaks

\numberwithin{equation}{section}

  \evensidemargin
\oddsidemargin
\newcommand{\ii}{{\rm{i}}}

\newcommand{\bS}{\mathbb{S}}

\def\dint{\textup{d}}

\newcommand{\E}{\mathbb E}

\newcommand{\R}{\mathbb{R}}
\newcommand{\N}{\mathbb{N}}

\newcommand{\C}{\mathbb{C}}
\newcommand{\Z}{\mathbb{Z}}

\renewcommand{\P}{\mathbb{P}}

\renewcommand{\Re}{\operatorname{Re}}
\renewcommand{\Im}{\operatorname{Im}}

\newcommand{\conv}{\mathop{\mathrm{conv}}\nolimits}
\newcommand{\pos}{\mathop{\mathrm{pos}}\nolimits}

\newcommand{\eps}{\varepsilon}

\newcommand{\bsl}{\backslash}

\newcommand{\dd}{{\rm d}}
\newcommand{\eee}{{\rm e}}

\theoremstyle{plain}
\newtheorem{theorem}{Theorem}[section]
\newtheorem{lemma}[theorem]{Lemma}

\newtheorem{proposition}[theorem]{Proposition}
\newtheorem{conjecture}[theorem]{Conjecture}
\newtheorem{open}[theorem]{Open question}

\theoremstyle{definition}

\newtheorem{example}[theorem]{Example}

\theoremstyle{remark}
\newtheorem{remark}[theorem]{Remark}

\begin{document}

\author{Zakhar Kabluchko}
\address{Zakhar Kabluchko: Institut f\"ur Mathematische Stochastik,
Westf\"alische Wilhelms-Universit\"at M\"unster,
Orl\'eans-Ring 10,
48149 M\"unster, Germany}
\email{zakhar.kabluchko@uni-muenster.de}

\title[Face numbers of high-dimensional Poisson zero cells]{Face numbers of high-dimensional Poisson zero cells}



\begin{abstract}
Let $\mathcal Z_d$ be the zero cell of a $d$-dimensional, isotropic and stationary Poisson hyperplane tessellation.
We study the asymptotic behavior of the expected number of $k$-dimensional faces of $\mathcal Z_d$, as $d\to\infty$. For example, we show that the expected number of hyperfaces of $\mathcal Z_d$ is asymptotically equivalent to $\sqrt{2\pi/3}\, d^{3/2}$, as $d\to\infty$.
 We also prove that the expected solid angle of a random cone spanned by $d$ random vectors that are independent and uniformly distributed on the unit upper half-sphere in $\mathbb R^{d}$ is asymptotic to $\sqrt 3  \pi^{-d}$, as $d\to\infty$.
\end{abstract}

\maketitle

\section{Main results}

\subsection{Random cones in a half-space}
Fix some $\ell\in \N_0$ and let $V_1,\ldots,V_{d+\ell}$ be $d+\ell$ random unit vectors drawn uniformly at random from the unit sphere $\bS^{d-1}$ in $\R^d$.
The convex cone spanned by these vectors, also known as their \textit{positive hull}, is denoted by
$$
D_{d+\ell, d} := \pos(V_1,\ldots,V_{d+\ell}) := \left\{\sum_{i=1}^{d+\ell} \lambda_i V_i: \lambda_1,\ldots,\lambda_{d+\ell}\geq 0\right\}.
$$
Explicit formulas for various geometric characteristics of the cones $D_{d+\ell,d}$ and several closely related objects have been derived by Wendel~\cite{Wendel}, Cover and Efron~\cite{cover_efron}, Donoho and Tanner~\cite{donoho_tanner1} and Hug and Schneider~\cite{HS15}. On the other hand, let $U_1,\ldots,U_{d+\ell}$ be random unit vectors sampled uniformly and independently from the \textit{upper half-sphere}
$$
\bS^{d-1}_+:=\{x = (x_1,\ldots,x_{d})\in \R^{d}:\, x_1\geq 0,\, x_1^2 +  \ldots + x_d^2 = 1\}.
$$
The convex cone  generated by these points was studied in~\cite{barany_etal,convex_hull_sphere,kabluchko_poisson_zero,kabluchko_simplified_formulas} and will be  denoted by
$$
C_{d+\ell, d} := \pos(U_1,\ldots,U_{d+\ell}) := \left\{\sum_{i=1}^{d+\ell} \lambda_i U_i: \lambda_1,\ldots,\lambda_{d+\ell}\geq 0\right\}.
$$
Let $\alpha(C)$ denote the solid angle of a convex cone $C$ normalized in such a way that the solid angle of the full space is $1$. Then, it can be deduced from Wendel's formula~\cite{Wendel}, see~\cite[Lemma~5.1]{godland_kabluchko_thaele_cover_efron}, that the expected solid angle of the cone $D_{d+\ell, d}$ is given by
\begin{equation}\label{eq:solid_angle_full_sphere_explicit}
\E \alpha (D_{d+\ell, d}) = \frac 1 {2^{d+\ell}} \sum_{j=d}^{d+\ell} \binom {d+\ell}{j}.
\end{equation}
For example, the expected value of the solid angle of $D_{d,d}$ equals $2^{-d}$ for symmetry reasons.
On the other hand, the expected solid angle of the random cone $C_{d+\ell, d}$ has been computed explicitly in~\cite{kabluchko_poisson_zero} and can be expressed through the numbers $A[n,k]$, indexed by $n\in\N_0$ and $k\in \Z$ with $k\leq n$, which may be defined by
\begin{equation}\label{eq:A_n_k_def}
A[n,k] := \frac{n!}{(n-k)!} \frac 1 {\pi} \int_{-\infty}^{+\infty} (\cosh x)^{-n-1} \left(\frac \pi 2 + \ii x\right)^{n-k} \dd x.
\end{equation}
This formula can be found in Section~6.9 and Eqn.~(6.8) (with $\alpha=1$) in~\cite{kabluchko_formula}; see also~\cite[Theorem~1.1]{kabluchko_poisson_zero} and~\cite[Section~2.2]{kabluchko_formula} for other, equivalent definitions of these numbers. For example, it follows from Theorem~2.5 in~\cite{kabluchko_poisson_zero} that
\begin{equation}\label{eq:E_alpha_C_d_d}
\E \alpha (C_{d,d}) = \frac{1}{2\cdot \pi^d} \frac{\sqrt \pi \, \Gamma(\frac{d+1}{2})}{\Gamma(\frac {d+2}{2})} d^2 A[d-1, -1].
\end{equation}

The asymptotic behavior of various quantities associated with $D_{d+\ell,d}$ and several closely related families of cones in the high-dimensional regime when $d\to\infty$ has been investigated in~\cite{donoho_tanner1,HugSchneiderThresholdPhenomena,godland_kabluchko_thaele_cover_efron,HugSchneiderThresholdPhenomena2}.
For example, \eqref{eq:solid_angle_full_sphere_explicit} easily yields that for every fixed $\ell\in \N_0$ we have\footnote{
Asymptotic equivalence of sequences is defined in the usual way: $a_d \sim b_d$ as $d\to\infty$  means that $\lim_{d\to\infty} a_d/b_d = 1$.}
\begin{equation}\label{eq:solid_angle_full_sphere_asympt}
\E \alpha (D_{d+\ell, d}) \sim \frac{(d/2)^\ell}{\ell!} 2^{-d},
\qquad
d\to\infty.
\end{equation}
In the present note, we are interested in the asymptotic behavior of the solid angle of $C_{d+\ell, d}$ and some related quantities,  as $d\to\infty$. The next theorem will be established by  means of  a saddle-point analysis of~\eqref{eq:A_n_k_def}.

\begin{theorem}\label{theo:asympt_A_n_k}
For every $k\in \Z$ we have\footnote{Note that for $k\in \{-2,-4,\ldots\}$ both sided of~\eqref{eq:A_n_k_asympt_main_theo} vanish because then $\Gamma(\frac k2 + 1) = \infty$ and $A[d,k] = 0$. The latter follows from Equations~(1.2) and~(1.3) in~\cite{kabluchko_poisson_zero}.}
\begin{equation}\label{eq:A_n_k_asympt_main_theo}
\lim_{d\to\infty} \frac{A[d,k]}{d^{3k/2}} = \frac{1}{6^{k/2} \Gamma(\frac k2 + 1)}.
\end{equation}
\end{theorem}
For example, Theorem~\ref{theo:asympt_A_n_k} combined with~\eqref{eq:E_alpha_C_d_d} easily implies that  $\E \alpha (C_{d, d}) \sim \sqrt 3 \,  \pi^{-d}$.
More generally, as a consequence of Theorem~\ref{theo:asympt_A_n_k} we shall derive the following


\begin{theorem}\label{theo:asympt_angle}
Let $\ell\in \N_0$ be fixed. Then,
\begin{equation}\label{eq:solid_angle_half_sphere_asympt}
\E \alpha (C_{d+\ell, d}) \sim \sqrt 3 \,\frac{(d/2)^{\ell}}{\ell!} \pi^{-d},
\qquad
d\to\infty.
\end{equation}
\end{theorem}

An application of Theorem~\ref{theo:asympt_A_n_k} to the analysis of sequential decision making algorithms can be found in~\cite{banerjee_halpern_peters}.





\begin{open}
Determine the asymptotics of the variance and the limit distribution of the solid angles of $C_{d,d}$ and $D_{d,d}$, as $d\to\infty$.
\end{open}

It has been shown in~\cite{kabluchko_poisson_zero} that the numbers $A[n,k]$ appear also in connection with several other objects including the zero cell of the $d$-dimensional, isotropic and stationary Poisson hyperplane tessellation, while certain more general numbers appear in connection with beta' polytopes and their limits, the so-called Poisson polyhedra.  In the following we shall study the expected number of faces of these random polytopes in the regime when $d\to \infty$ continuing the line of research which was initiated in~\cite{HoermannHugReitznerThaele}.
Let us define the objects we are interested in.

\subsection{Poisson zero cells}
Consider a unit intensity Poisson point process  $\{T_i\}_{i\in\Z}$ on the real line. Independently, let $\{V_i\}_{i\in\Z}$  be independent, identically distributed  random vectors  distributed uniformly on the unit sphere $\bS^{d-1}$ in $\R^d$.  A (stationary and isotropic) \textit{Poisson hyperplane process} with unit intensity is a random, countable collection $\{H_i\}_{i\in \Z}$ of affine hyperplanes given by
$$
H_i:= \{x\in \R^d: \langle x, V_i \rangle = T_i\}, \qquad i\in \Z.
$$
Note that the $V_i$'s determine the normal directions of the hyperplanes, while the distances from the hyperplanes to the origin are given by the absolute values of the $T_i$'s.

The hyperplanes dissect $\R^d$ into countably many cells which constitute the so-called \textit{Poisson hyperplane tessellation}. Its law stays invariant under isometries of $\R^d$. The \textit{Poisson zero polytope} $\mathcal Z_d$ is the a.s.\ unique polytope of this tessellation that contains the origin.
The expected number of faces of $\mathcal Z_d$ has been determined in~\cite[Theorem~1.1]{kabluchko_poisson_zero} as follows: For all $d\in \N$ and $\ell\in \{0,\ldots,d\}$ we have
\begin{equation}\label{eq:theo:main}
\E f_\ell(\mathcal Z_d) = \frac{\pi^{d-\ell}}{(d-\ell)!} A[d,d-\ell],
\end{equation}
where $f_\ell(P)$ denotes the number of $\ell$-dimensional faces of a polytope $P$ and the numbers $A[n,k]$ are the same as in~\eqref{eq:A_n_k_def}. Combined with Theorem~\ref{theo:asympt_A_n_k}, this formula yields the following
\begin{theorem}\label{theo:exp_f_poi_zero_cell_fixed_k}
For every $k\in \N$ we have
\begin{equation}\label{eq:Ef_k_alpha_1}
\E f_{d-k} (\mathcal Z_d)
\sim  \frac{d^{3k/2}}{k! \Gamma\left(\frac k2 + 1\right)} \left(\frac{\pi}{\sqrt 6}\right)^k,
\qquad
d\to\infty.
\end{equation}
\end{theorem}



\subsection{Poisson polyhedra}
For a parameter $\alpha>0$ let $\Pi_{d,\alpha}$ be a Poisson point process on $\R^d\backslash\{0\}$ with the following  power-law intensity (w.r.t.\ the Lebesgue measure):
$$
x\mapsto \|x\|^{-d-\alpha},\qquad x\in \R^d \bsl \{0\}.
$$
With probability $1$, the total number of atoms of $\Pi_{d,\alpha}$ is infinite, and the atoms accumulate at the origin (because the intensity is not integrable at $0$). On the other hand, the set of atoms is a.s.\ bounded (because the intensity is integrable at $\infty$).
The convex hull of the atoms of $\Pi_{d,\alpha}$, denoted by $\conv \Pi_{d,\alpha}$, will be referred to as the \textit{Poisson polyhedron}. It is known~\cite{convex_hull_sphere}  that $\conv \Pi_{d,\alpha}$ is almost surely a polytope despite of being defined as a convex hull of an \textit{infinite} set. Moreover, $\conv \Pi_{d,\alpha}$ is a.s.\ a \textit{simplicial polytope} meaning that all of its faces are simplices. Poisson polyhedra (and their dual objects, the zero cells of certain non-stationary Poisson hyperplane tessellations with power-law intensity) appeared in~\cite{HugSchneider07LargeCells,hoermann_hug,HoermannHugReitznerThaele,convex_hull_sphere,beta_polytopes,kabluchko_formula}.
It is known, see, e.g.~\cite[Theorem 1.23]{beta_polytopes} that, up to rescaling, $\conv\Pi_{d,1}$ (with $\alpha=1$) has the same distribution as $\mathcal Z_d^\circ:= \{x\in \R^d: \langle x,y\rangle \leq 1 \text{ for all } y\in \mathcal Z_d\}$,  the convex dual of $\mathcal Z_d$. In particular, we have
\begin{equation}\label{eq:duality}
\E f_k(\mathcal Z_d) = \E f_{d-k-1}(\conv\Pi_{d,1}),
\qquad
k\in \{0,\ldots, d-1\}.
\end{equation}
An explicit formula for the expected number of $k$-dimensional faces of $\conv \Pi_{d,\alpha}$, for any $k\in\{0,\ldots,d-1\}$ and $\alpha>0$, has been derived in~\cite{kabluchko_formula} and~\cite{beta_polytopes}. An asymptotic analysis of this formula, which will be carried out in Section~\ref{subsec:proof_Ef_poi_polyhedron_fixed_k}, yields the following result.
\begin{theorem}\label{theo:Ef_Poi_poly_fixed_k}
Let $\alpha>0$ and $k\in \N$ be fixed. Then, as $d\to\infty$, we have
$$
\E f_{k-1}(\conv \Pi_{d,\alpha}) \sim  \frac{d^{k(1+\frac \alpha 2)}}{k! \Gamma \left(\frac {\alpha k }{2} + 1\right)} \left(\frac{\alpha \sqrt \pi\,   \Gamma (\frac \alpha 2)}{\Gamma (\frac{\alpha+1}{2})(2 + \frac 4 \alpha)^{\alpha/2}  }\right)^k.
$$
\end{theorem}
\begin{example}
For $\alpha = 1$ we recover Theorem~\ref{theo:exp_f_poi_zero_cell_fixed_k}.
For $\alpha = 2$, Theorem~2.7 in~\cite{kabluchko_formula} gives the exact formula
$$
\E f_{k-1}(\conv \Pi_{d,2}) = \binom dk \binom{d+k}{k} \sim \frac{d^{2k}}{(k!)^2},
$$
which is consistent with the above asymptotics.
\end{example}
\begin{remark}
In~\cite[Theorem~3.22(i)]{HoermannHugReitznerThaele}, a positive upper bound on $\limsup_{d\to\infty} \sqrt[d]{\E f_{k-1}(\conv \Pi_{d,\alpha})}$ was established. In fact, Theorem~\ref{theo:Ef_Poi_poly_fixed_k} implies that this $\limsup$ is a limit and  equals zero.
\end{remark}
\begin{open}
The asymptotics of the variance of $\E f_{k-1}(\conv \Pi_{d,\alpha})$ remains unknown. In particular, we do not know whether the $f$-vector is concentrated around its expectation, i.e.\ whether
\begin{equation}\label{eq:concentration_conj}
\frac{f_{k-1}(\conv \Pi_{d,\alpha})}{\E f_{k-1}(\conv \Pi_{d,\alpha})}  \overset{P}{\underset{d\to\infty}\longrightarrow} 1
\qquad
\text{ in probability.}
\end{equation}
\end{open}


\subsection{Speculations on neighborliness}
A polytope $P$ is called \textit{$k$-neighborly} if the convex hull of every $k$ vertices of $P$ defines a $(k-1)$-dimensional face of $P$; see~\cite[Chapter~7]{GruenbaumBook}. In other words, a polytope with $n$ vertices is $k$-neighborly if $f_{k-1}(P)$ attains its maximal possible value, i.e.\ $\binom nk$. In the work of Vershik and Sporyshev~\cite{vershik_sporyshev_asymptotic_faces_random_polyhedra1992} which was continued in a series of papers by Donoho and Tanner~\cite{donoho_neighborliness_proportional,donoho_tanner_neighborliness,donoho_tanner,donoho_tanner1}, it has been shown that \textit{random} polytopes have surprisingly strong neighborliness properties. These authors studied Gaussian polytopes (i.e.\ convex hulls of i.i.d.\ Gaussian random samples) or, which is essentially the same, random projections of regular simplices. More recently, a similar analysis has been performed for convex hulls of random walks~\cite{kabluchko_marynych_lah}.

Let us now make some speculations on the neighborliness properties of the polytopes $\conv \Pi_{d,\alpha}$. For simplicity, we restrict ourselves to the case $\alpha = 1$.  Taking $k=1$ in Theorem~\ref{theo:exp_f_poi_zero_cell_fixed_k} and recalling~\eqref{eq:duality} we obtain
$$
\E f_{0}(\conv \Pi_{d,1}) = \E f_{d-1}(\mathcal Z_d) \sim  \sqrt{\frac {2\pi}{3}} d^{3/2}.
$$
This  means that for arbitrary $k\in \N$ we can rewrite~\eqref{eq:Ef_k_alpha_1} as follows:
$$
\E f_{k-1}(\conv \Pi_{d,1}) \sim  \frac{(\E f_{0}(\conv \Pi_{d,1}))^k}{k!} \cdot p_k
\quad
\text{ with }
\quad
p_k := \frac{\pi^{k/2}}{2^k \Gamma(\frac k 2 + 1)}.
$$
If we ignore the expectations, then the first factor on the right-hand side would be the number of unordered $k$-tuples of vertices of $f_{0}(\conv \Pi_{d,1})$, up to asymptotic equivalence. By a strange coincidence, the constant $p_k$ on the right-hand side equals the volume of the $k$-dimensional ball of radius $1/2$. Since this ball can be inscribed into the unit cube, we have $p_k < 1$ for all $k\geq 2$. In particular, we conjecture that the polytope $\Pi_{d,\alpha}$ is not $k$-neighborly with probability converging to $1$ as $d\to\infty$.

\begin{conjecture}
Fix $k\in \{2,3,\ldots\}$ and sample $k$ vertices uniformly at random from the set of vertices of the random polytope $\conv \Pi_{d,1}$. Then, the probability that the simplex spanned by these vertices is a $(k-1)$-dimensional face of $\conv \Pi_{d,1}$ converges, as $d\to\infty$, to some limit $q_k\in (0,1)$. Moreover, if~\eqref{eq:concentration_conj} holds, we conjecture that $p_k = q_k$.
\end{conjecture}


\subsection{Asymptotics in the linearly growing \texorpdfstring{$k$}{k} regime}
To motivate what follows, let us recall that the $f$-vectors of the $d$-dimensional cube $[0,1]^d$ and the $d$-dimensional simplex $S_d$  are given by
$$
f_k([0,1]^d)
=
2^{d-k} \binom{d}{k}
,
\qquad
f_k(S_d)
=
\binom {d+1}{k+1},
\qquad k\in \{0,\ldots,d\}.
$$
Stirling's formula implies that for $k\sim \lambda d$ with fixed $\lambda \in (0,1)$ we have
\begin{align*}
&\lim_{d\to\infty} \frac 1d \log f_k ([0,1]^d) = (\log 2)(1-\lambda) - \lambda \log \lambda - (1-\lambda)\log (1-\lambda),
\\
&\lim_{d\to\infty} \frac 1d \log  f_k (S_d) = - \lambda \log \lambda - (1-\lambda)\log (1-\lambda).
\end{align*}
Vershik and Sporyshev~\cite{vershik_sporyshev_asymptotic_faces_random_polyhedra1992} proved a similar result, with a different limit function,  for expected face numbers of Gaussian polytopes. It is natural to conjecture that many other random and non-random high-dimensional polytopes should exhibit a similar type of behavior.  The next result confirms this for Poisson zero cells.
\begin{theorem}\label{theo:Ef_Poi_poly_linear_k}
Let $k= k(d) \sim \lambda d$ as $d\to\infty$, where $\lambda\in (0,1)$ is fixed. Then,
\begin{multline}\label{eq:exp_profile_poi_poly}
\lim_{d\to\infty} \frac 1d \log \E f_{k-1} (\conv \Pi_{d,1})
=
\lim_{d\to\infty} \frac 1d \log \E f_{d-k} (\mathcal Z_d)
\\
=
\lambda \log \pi - (\lambda \log \lambda + (1-\lambda)\log (1-\lambda)) + (1-\lambda) \log \psi(\lambda) - \log \sin \psi(\lambda),
\end{multline}
where $\psi(\lambda)$ is the unique solution to the equation $1-\lambda = \psi(\lambda) \cot \psi(\lambda)$ in the interval $(0,\frac \pi 2)$.
\end{theorem}
We do not know whether the expectations can be removed in the above result, i.e.\ whether $\frac 1d \log f_{k-1} (\conv \Pi_{d,1})$ converges to the same limit in probability. The plot of the ``exponential profile'' appearing on the right-hand side of~\eqref{eq:exp_profile_poi_poly} is shown on Figure~\ref{fig:exp_profile}. The maximum of the profile is attained at $\lambda_* \approx 0.699155$. Hence, on average, most faces of the high-dimensional Poisson zero cell have dimensions close to $\lambda_* d$, but it remains unclear whether the same conclusion remains true with high probability. Theorem~\ref{theo:Ef_Poi_poly_linear_k} could be generalized to $\conv \Pi_{d,\alpha}$ with arbitrary $\alpha>0$, but the limit function is ugly and we refrain from stating this generalization.

\begin{figure}[t]
	\centering
	\includegraphics[width=0.4\columnwidth]{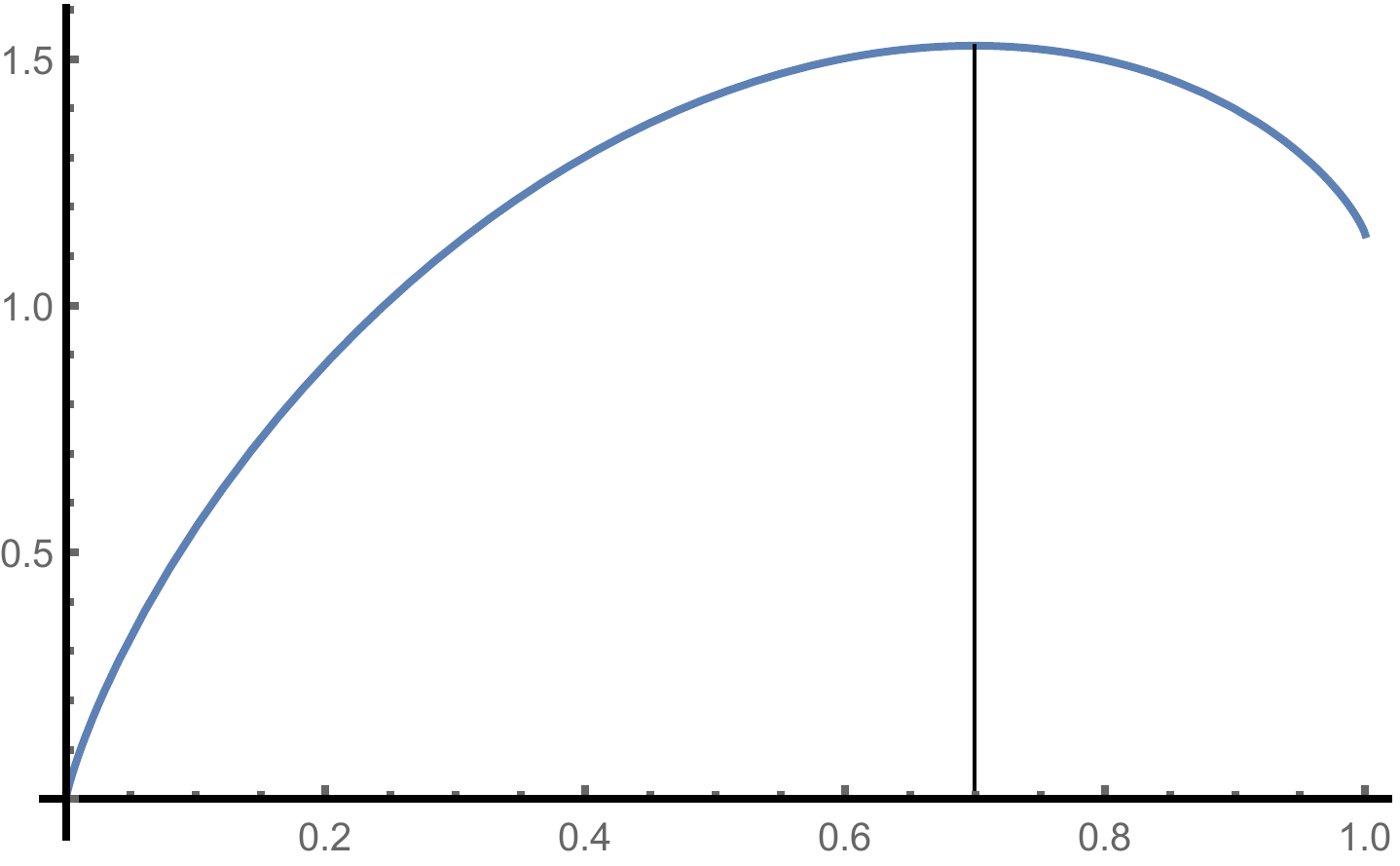}
	\caption{The function on the right-hand side of~\eqref{eq:exp_profile_poi_poly} and its maximizer.}
\label{fig:exp_profile}
\end{figure}

\subsection{Asymptotics in the regime with fixed \texorpdfstring{$d-k$}{d-k}}
Let us finally briefly mention results which are known in the regime when $k= d-\ell$ with fixed $\alpha>0$ and $\ell \in \N$.
Theorems~1.2 and~3.21(ii) of~\cite{HoermannHugReitznerThaele} state that
$$
\lim_{d\to\infty}\sqrt[d]{\E f_{d-\ell}(\conv \Pi_{d,\alpha})} = \frac {\sqrt{\pi}\alpha\Gamma(\frac \alpha 2)}{\Gamma(\frac{\alpha+1}{2})}.
$$
A more refined result, obtained in Theorems 1.23 and 1.25 of~\cite{beta_polytopes}, states that
$$
\E f_{d-\ell-1} (\conv \Pi_{d,\alpha}) \sim
\frac {\sqrt{\alpha}} {2^{\ell-{\frac 1 2}}} \left( \frac{\Gamma(\frac \alpha 2)} {\Gamma(\frac{\alpha+1}{2})}\right)^{d} \frac {(\sqrt{\pi}\,\alpha)^{d-1}} {\ell!} d^{\ell-{\frac 1 2}},
\qquad
d\to\infty.
$$
It follows that
$$
\frac{\E f_{d-\ell-1} (\conv \Pi_{d,\alpha})}{\E f_{d-1} (\conv \Pi_{d,\alpha})} \sim  \frac{(d/2)^\ell}{\ell!},
\qquad
d\to\infty.
$$
The same behavior is exhibited by the $f$-vector of a $d$-dimensional crosspolytope $C_d$, given by $f_{d-\ell-1}(C_d) =  2^{d-\ell} \binom d\ell$, although we have no explanation for this coincidence.

\subsection{The typical cell of the Poisson hyperplane tessellation}
Let us conclude by stating some conjectures on the \textit{typical cell} of the isotropic and stationary Poisson hyperplane tessellation, which is a rare example for which explicit results on \textit{second} moments are available thanks to a formula proved in the paper of Schneider~\cite[Theorem 1.2]{schneider_second_moments} (which also contains a discussion of the history of this formula with a reference to the work of Miles~\cite{miles_thesis,miles_synopsis}).
Informally, the typical cell is a random polytope picked ``uniformly at random'' from the infinite collection of cells of the Poisson hyperplane tessellation, where all cells have the same chance to be picked regardless of their volume.  For a precise definition using the Palm calculus we refer to~\cite{SW08}. The zero cell studied above is a volume-biased version of the typical cell~\cite[Theorem 10.4.1]{SW08}.
It is known~\cite[Theorems~10.3.1 and~10.3.2]{SW08} that the expected $f$-vector of $\mathcal Z_d^{\text{typ}}$ is the same as the $f$-vector of the cube, that is
$$
\E f_k(\mathcal Z_d^{\text{typ}}) = 2^{d-k} \binom{d}{k},
\qquad
k\in \{0,\ldots,d\}.
$$
Regarding the second moment of the number of vertices, Schneider~\cite{schneider_second_moments} (following  Miles~\cite{miles_thesis,miles_synopsis}) proved that
\begin{equation}\label{eq:sec_moment_explicit}
\E f_0^2(\mathcal Z_d^{\text{typ}}) = 2^d d! \sum_{j=0}^d \frac{\kappa_j^2}{4^j (d-j)!},
\end{equation}
where $\kappa_j := \pi^{j/2}/\Gamma (\frac j2 +1)$ is the volume of the $j$-dimensional unit ball. An asymptotic analysis of this expression as $d\to\infty$ will be carried out in Section~\ref{subsec:proof_asympt_sec_moments} and yields the following
\begin{proposition}\label{prop:asympt_sec_moment}
We have $\lim_{d\to\infty} \frac 1d \log \E f_0^2(\mathcal Z_d^{\text{typ}}) = \log (\pi + 2)$.
\end{proposition}

So, we have $\E f_0(\mathcal Z_d^{\text{typ}}) = 2^d$ and  $\E f_0^2(\mathcal Z_d^{\text{typ}}) = (\pi + 2 + o(1))^d$. This allows us to speculate about the possible distributional limit behavior of $f_0(\mathcal Z_d^{\text{typ}})$ as $d\to\infty$. If we assume that $f_0(\mathcal Z_d^{\text{typ}}) / \E f_0(\mathcal Z_d^{\text{typ}})$ converges  in distribution to $1$ (or, more generally, to some random variable $W>0$) together with all moments, then we would have $\E f_0^2(\mathcal Z_d^{\text{typ}}) \sim \text{const} \cdot (\E f_0(\mathcal Z_d^{\text{typ}}))^2$, which is evidently a contradiction.
On the other hand, the following type of behaviour  is consistent with the known results on the first two moments.

\begin{conjecture}\label{conj:LDP}
As $d\to\infty$, the sequence of random variables $(\frac 1d \log f_0(\mathcal Z_d^{\text{typ}}))_{d\in \N}$ satisfies a large deviation principle on $\R$ with speed $d$ and certain non-degenerate rate function.
\end{conjecture}

Indeed, if we denote the rate function by $J(x)$, then the large deviation principle states that, informally speaking,
$$
\P[\log f_0(\mathcal Z_d^{\text{typ}}) \approx d \cdot x] \approx \eee^{-d J(x)},
\qquad
x\geq 0.
$$
This suggests that for every $m>0$, the $m$-th moment of $f_0(\mathcal Z_d^{\text{typ}})$ should satisfy
$$
\E f_0^m (\mathcal Z_d^{\text{typ}}) = \E \eee^{md \cdot \frac 1d \log f_0(\mathcal Z_d)} \approx \eee^{d J^*(m)},
$$
where $J^*(m) = \sup_{x\geq 0} (mx - J(x))$ is the Legendre-Fenchel transform of $J$. The special cases $m=1$ and $m=2$ discussed above suggest that   $J^*(1)=\log 2$ and $J^*(2) = \log (\pi + 2)$. A behaviour similar to the one suggested in Conjecture~\ref{conj:LDP} is known for the random energy model~\cite{fedrigo_etal} and is accompanied by a similar behaviour of moments. Conjecture~\ref{conj:LDP} naturally suggests a central  limit theorem of the form
$$
\frac {\log f_0(\mathcal Z_d^{\text{typ}}) - \E \log f_0(\mathcal Z_d^{\text{typ}})}{\sqrt d}
{\overset{d}{\underset{d\to\infty}\longrightarrow}}
 N(0,\sigma^2),
$$
where $\sigma^2$ is certain unknown variance.

\section{Proofs}
\subsection{Proof of Theorem~\ref{theo:Ef_Poi_poly_fixed_k}}\label{subsec:proof_Ef_poi_polyhedron_fixed_k}
The starting point of the proof is the following explicit formula which can be found in Equation~(6.32) of~\cite{kabluchko_formula}:
\begin{equation}\label{eq:E_f_k_formula_explicit}
\E f_{k-1}(\conv \Pi_{d,\alpha})
=
\alpha^d \binom dk \left( \frac{\sqrt \pi \Gamma (\frac \alpha 2)}{\Gamma(\frac {\alpha+1}{2})} \right)^{k} T_{d,k}(\alpha).
\end{equation}
Here, the term $T_{d,k}(\alpha)$ is given by
\begin{equation}\label{eq:T_d_k_asympt}
T_{d,k}(\alpha) =  \frac 1 \pi \int_{-\infty}^{+\infty} \frac{\tilde F(\ii u)^{d-k}}{(\cosh u)^{\alpha d + 1}} \dint u,
\end{equation}
where, by Equation~(6.1) of~\cite{kabluchko_formula},
$$
\tilde F(\ii u) = \int_{-\pi/2}^{\ii u} (\cos y)^{\alpha-1} \dint y,
\qquad
\ii u \in \mathcal D,
$$
and $\mathcal D$ is the complex plane with two cuts at $(-\infty, -\frac \pi 2]$ and $[+\frac \pi 2, +\infty)$.
The integral in the definition of the function $\tilde F$ is taken over some contour connecting $-\frac \pi 2$ to $\ii u$ and staying in the domain $\mathcal D$.  Note that the ramification points of the multivalued function $(\cos y)^{\alpha-1}$ are located at $\frac \pi 2 + \pi n$, $n\in \Z$,  hence this function has a well-defined branch  in $\mathcal D$ characterized by the condition $(\cos 0)^{\alpha-1} = 1$.
We start by making the substitution $u=-\ii v + \ii \frac \pi 2$ (meaning that  $v = \ii u + \frac {\pi}{2}$), which transforms the integral to
\begin{equation}\label{eq:int_T_d_k}
T_{d,k}(\alpha) = \frac {1}{\pi \ii}\int_{-\ii \infty + \frac {\pi}{2}}^{+\ii \infty + \frac {\pi}{2}} \frac{\left(\int_{0}^{v}(\sin z)^{\alpha-1}\dint z\right)^{d-k}}{(\sin v)^{\alpha d + 1}} \dint v
\end{equation}
because $\cosh u = \cosh (-\ii v + \ii \frac {\pi}{2}) = \sin v$ and
$$
\tilde F(\ii u)
=
\tilde F \left(v-\frac \pi 2\right)
=
\int_{-\pi/2}^{v-\pi/2} (\cos z)^{\alpha-1} \dint z
=
\int_{0}^{v}(\sin z)^{\alpha-1}\dint z.
$$
We use the branch of $(\sin z)^{\alpha-1}$ in $\C$ with slits at $(-\infty, 0]$ and $[\pi,\infty)$ characterized by $(\sin \frac \pi 2)^{\alpha-1} = 1$.

In the following we shall study the asymptotics of the integral~\eqref{eq:int_T_d_k} as $d\to\infty$.
Shifting the integration contour from $(-\ii \infty + \frac {\pi}{2}, +\ii \infty + \frac {\pi}{2})$ to a contour $\gamma_{1/\sqrt d}$, where $\gamma_{c}$ with $c>0$ is a contour which starts at $-\ii\infty$, proceeds along the negative imaginary half-axis to $-c \ii$, makes a counterclockwise half-loop in the right half-plane to $+c \ii$ (to avoid the singularity at $0$), and then goes to $+\ii \infty$ along the positive imaginary half-axis, we obtain
\begin{align}
T_{d,k} (\alpha)
&=
\frac {1}{\pi \ii} \int_{\gamma_{1/\sqrt d}} \frac{\left(\int_{0}^{v}(\sin z)^{\alpha-1}\dint z\right)^{d-k}}{(\sin v)^{\alpha d + 1}} \dint v\notag\\
&=
\frac {1}{\pi \ii}\int_{\gamma_{1/\sqrt d}} \left( \frac{\int_{0}^{v}(\sin z)^{\alpha-1}\dint z}{(\sin v)^{\alpha}}\right)^{d} \cdot \frac{\dint v}{\left(\int_{0}^{v}(\sin z)^{\alpha-1}\dint z\right)^k \cdot \sin v}.\label{eq:int_T_d_k_1}
\end{align}
Note that the contour shift is legitimate because the function under the sign of the integral decays exponentially as $\Im z\to \pm \infty$ and uniformly as long as $\Re z \in [0,\frac \pi 2]$. This is easy to check, for sufficiently large $d$, by considering the cases $\alpha >1$, $\alpha=1$ and $0<\alpha<1$ separately.
Next we are going to compute the asymptotics of the function which is raised to $d$-th power in~\eqref{eq:int_T_d_k_1}, in the regime when $v\to 0$. First of all, as $z\to 0$ (while staying in the strip $\Re z\in [0,\pi]$, $z\neq 0$), we have
$$
(\sin z)^{\alpha-1} = z^{\alpha-1} \left(\frac {\sin z}{z}\right)^{\alpha-1} = z^{\alpha-1} \left(1 + \frac {1-\alpha}{6}z^2 + O(z^4)\right).
$$
Integrating this, we obtain, as $v\to 0$,
\begin{equation}\label{eq:int_sin_asympt}
\int_{0}^{v}(\sin z)^{\alpha-1}\dint z
=
\frac {v^{\alpha}}{\alpha} + \frac{1-\alpha}{6(\alpha+2)} v^{\alpha+2} + O(v^{\alpha + 4})
=
v^{\alpha} \left(\frac 1 \alpha + \frac {1-\alpha}{6(\alpha + 2)} v^2 + O(v^4)\right).
\end{equation}
Also note that, as $v\to 0$,
\begin{equation}\label{eq:sin_v_alpha_taylor}
(\sin v)^\alpha = v^{\alpha} \left(\frac {\sin v}{v}\right)^\alpha = v^{\alpha} \left(1 - \frac {\alpha}{6} v^2 + O(v^4)\right).
\end{equation}
Combining~\eqref{eq:int_sin_asympt} and~\eqref{eq:sin_v_alpha_taylor}, we obtain
\begin{equation}\label{eq:tech1}
\frac{\int_{0}^{v}(\sin z)^{\alpha-1}\dint z}{(\sin v)^{\alpha}}
=
\frac{\frac 1 \alpha + \frac {1-\alpha}{6(\alpha + 2)} v^2 + O(v^4)}{1 - \frac {\alpha}{6} v^2 + O(v^4)}
=
\frac 1 \alpha \left(1 + \frac{\alpha }{2\alpha + 4} v^2 + O(v^4)\right).
\end{equation}
Now we apply to~\eqref{eq:int_T_d_k_1} the following  substitution:
$$
v = \frac{w}{\sqrt d} \sqrt {\frac {\alpha+2}\alpha}.
$$
Then, \eqref{eq:tech1} implies that
$$
\left( \frac{\int_{0}^{v}(\sin z)^{\alpha-1}\dint z}{(\sin v)^{\alpha}}\right)^{d}
= \frac 1 {\alpha^d} \left(1 + \frac {w^2}{2d} + O(d^{-2}) \right)^d
\sim
\frac {\eee^{w^2/2}} {\alpha^d},
\qquad
d\to\infty
.
$$
Also, \eqref{eq:int_sin_asympt} implies that, as $v\to 0$,
$$
\left(\int_{0}^{v}(\sin z)^{\alpha-1}\dint z\right)^k \cdot \sin v
\sim
\alpha^{-k} \left(\frac{w}{\sqrt d} \sqrt {\frac {\alpha+2}\alpha}\right)^{\alpha k + 1}.
$$
Altogether, we obtain, as $d\to\infty$,
$$
\left( \frac{\int_{0}^{v}(\sin z)^{\alpha-1}\dint z}{(\sin v)^{\alpha}}\right)^{d} \cdot \frac{1}{\left(\int_{0}^{v}(\sin z)^{\alpha-1}\dint z\right)^k \cdot \sin v} \sim
\alpha^{k-d} \eee^{w^2/2} \left(\frac{w}{\sqrt d} \sqrt {\frac {\alpha+2}\alpha}\right)^{-\alpha k - 1}.
$$
In order to apply the saddle point method we need to check that the maximum of the function which we raise to the $d$-th power on the imaginary axis is attained at $v = 0$. This is done in the following
\begin{lemma}
Let $v= \ii y$ with $y\neq 0$. Then,
$$
\left|\frac{\int_0^v (\sin z)^{\alpha-1}}{(\sin v)^\alpha}\right| < \frac 1 \alpha.
$$
\end{lemma}
\begin{proof}
Without restriction of generality let $y>0$. Using the substitutions $z:= \ii u$ and then $t:=\sinh u$, we obtain
$$
\left|\frac{\int_0^v (\sin z)^{\alpha-1}\dint z}{(\sin v)^\alpha}\right|
=
\frac{\int_0^y (\sinh u)^{\alpha-1}\dint u}{(\sinh y)^\alpha}
=
\frac{\int_0^{\sinh y} t^{\alpha-1}\frac{\dint t}{\sqrt{1+t^2}}}{(\sinh y)^\alpha}
<
\frac{\int_0^{\sinh y} t^{\alpha-1} \dint t}{(\sinh y)^\alpha}
= \frac 1 \alpha,
$$
which proves the claim.
\end{proof}
Knowing this, we can apply the standard saddle point method as in~\cite[Section~45]{sidorov_fedoryuk_shabunin_book},  which yields
$$
\lim_{d\to\infty} \frac{T_{d,k}(\alpha)}{(\alpha d)^{\alpha k /2}}
=
\frac {1}{\pi \ii}  \frac{\alpha^{k-d} }{(\alpha+2)^{\alpha k /2}} \int_{\gamma_{\sqrt{\alpha/(\alpha+2)}}} \eee^{w^2/2} w^{-\alpha k - 1} \dint w.
$$
To compute the remaining integral, we use the substitution $z =  - w^2/2$ meaning that $w= \sqrt{-2z}$, where we take the branch of $\sqrt z$ which is analytic on $\C\backslash (-\infty,0)$. This results in
\begin{multline*}
\frac 1 {\pi \ii} \int_{\gamma_{\sqrt{\alpha/(\alpha+2)}}}  \eee^{w^2/2} w^{-\alpha k - 1} \dint w
=
- \frac 1 {\pi \ii} \int_C \eee^{-z} (\sqrt {-2z})^{-\alpha k - 1} \frac {\dint z}{\sqrt{-2z}}
\\
=
\frac { \ii}  {2 \pi 2^{\frac{\alpha k}2}} \int_C \eee^{-z} (-z)^{-\frac{\alpha k + 2}2} \dint z
=
\frac{1}{2^{\frac{\alpha k} 2} \Gamma\left(\frac {\alpha k+2}{2}\right)},
\end{multline*}
where the contour $C$  starts at $+\infty$ on the real axis, encircles the origin in the counterclockwise direction and returns to $+\infty$. In the last step we used the  well-known~\cite[12.22, p.~245]{whittaker_watson} Hankel formula
$$
\frac 1 {\Gamma(x)} = \frac {\ii}{2\pi} \int_C \eee^{-z} (-z)^{-x} \dint z,
\qquad x\in \C.
$$
It follows that
\begin{equation}\label{eq:T_d_k_asymptotics}
\lim_{d\to\infty} \frac{T_{d,k}(\alpha)}{(\alpha d)^{\alpha k /2}} =  \frac{\alpha^{k-d}  }{(\alpha + 2)^{\alpha k/2}2^{\alpha k/ 2 } \Gamma\left(\frac {\alpha k+2}{2}\right)}.
\end{equation}
Inserting this into~\eqref{eq:E_f_k_formula_explicit} completes the proof of Theorem~\ref{theo:Ef_Poi_poly_fixed_k}. \hfill $\Box$

\subsection{Proof of Theorem~\ref{theo:asympt_A_n_k}}
We just have to take $\alpha = 1$ in the previous proof. Indeed, it follows from~\eqref{eq:A_n_k_def} and~\eqref{eq:T_d_k_asympt} that
$$
A[d,k] = \frac{d!}{(d-k)!} T_{d,k}(1).
$$
The claim follows from~\eqref{eq:T_d_k_asymptotics} with $\alpha = 1$. Note that the proof of~\eqref{eq:T_d_k_asymptotics} is valid for all $k\in \Z$. \hfill$\Box$

\subsection{Proof of Theorem~\ref{theo:Ef_Poi_poly_linear_k}}
By~\eqref{eq:theo:main} and~\eqref{eq:A_n_k_def} we have
\begin{align}
\E f_{d-k} (\mathcal Z_d)
&=
\frac{\pi^k}{k!} A[d,k] \notag\\
&=
2\cdot \pi^k \binom dk  \frac 1 {2\pi} \int_{-\infty}^{+\infty} (\cosh x)^{-d-1} \left(\frac \pi 2 + \ii x\right)^{d-k} \dd x \notag\\
&=
2\cdot \pi^k \binom dk \frac 1 {2\pi \ii} \int_{-\ii \infty + \frac \pi 2}^{\ii \infty + \frac \pi 2} \frac {z^{d-k}}{(\sin z)^{d+1}} \dint z,
\label{eq:proof_central_explicit}
\end{align}
where in the last line we used the substitution $z:= \frac \pi 2 + \ii x$.
Recall that $k\sim \lambda d$ as $d\to\infty$. Stirling's formula entails that
\begin{equation}\label{eq:proof_central_1}
\lim_{d\to\infty} \frac 1d \log \left(2\cdot \pi^k \binom dk\right) = \lambda \log \pi - (\lambda \log \lambda + (1-\lambda)\log (1-\lambda)).
\end{equation}
The main difficulty is to treat the remaining integral. We write it as
$$
\frac 1 {2\pi \ii} \int_{-\ii \infty + \frac \pi 2}^{\ii \infty + \frac \pi 2} \frac {z^{d-k}}{(\sin z)^{d+1}} \dint z
=
\frac 1 {2\pi \ii} \int_{-\ii \infty + \frac \pi 2}^{\ii \infty + \frac \pi 2} \left(f_d(z)\right)^d \frac{\dint z}{\sin z},
$$
where
$$
f_d(z) := \frac{z^{\frac{d-k}{d}}}{\sin z} \;\; {\underset{d\to\infty}\longrightarrow} \;\; \frac{z^{1-\lambda}}{\sin z} =: f_\infty(z).
$$
Note that the multivalued functions $z\mapsto z^{\frac{d-k}{d}}$ and $z\mapsto z^{1-\lambda}$ have well-defined branches in $\C\backslash (-\infty, 0)$ characterized by requiring them to be real for $z>0$.

Now we are going to shift the integration contour so that we can apply the saddle point method~\cite[Section~45]{sidorov_fedoryuk_shabunin_book}.
Taking the logarithmic derivative, one checks that on the interval $(0,\frac \pi2)$ the function $x\mapsto f_d(x)$ has a unique minimizer at $x = \psi(k/d)$, where we recall that for $\mu\in (0,1)$ we denote by $\psi(\mu)$ the unique solution to
$$
1-\mu = \psi(\mu) \cot \psi(\mu), \qquad 0 < \psi(\mu) <\pi/2.
$$
Existence and uniqueness of the solution follow from the fact that the function $y\mapsto y \cot y$ is continuous and decays from $1$ to $0$ on the interval $(0,\frac \pi 2)$.   Note that $\lim_{d\to\infty}\psi(k/d) = \psi(\lambda) \in (0,\frac \pi 2)$.  We now shift the contour of integration to the vertical line passing through $\psi(k/d)$, that is we write
$$
\frac 1 {2\pi \ii} \int_{-\ii \infty + \frac \pi 2}^{\ii \infty + \frac \pi 2} \frac {z^{d-k}}{(\sin z)^{d+1}} \dint z
=
\frac 1 {2\pi \ii} \int_{-\ii \infty + \psi(k/d)}^{\ii \infty + \psi(k/d)}  \left(f_d(z)\right)^d  \frac{\dint z}{\sin z}.
$$
This is legitimate because $\sin z$ increases exponentially fast as $\Im z\to \pm \infty$ (uniformly in $\Re z$).
To apply the saddle point method, we need to verify that,  on the contour of integration,  the maximum of the function $|f_d(z)|$ is attained at $z=\psi(k/d)$. This is done in the following
\begin{lemma}\label{lem:saddle_point_justify}
Let $z= x+iy$ with $0 < x < \pi$ and $y\in \R$. Then, for every $\mu \in (0,1)$ we have
$$
\left|\frac{z^{1-\mu}}{\sin z}\right| \leq \frac{x^{1-\mu}}{\sin x}.
$$
The inequality is strict if $y\neq 0$.
\end{lemma}
\begin{proof}
In view of the product formula for the sine function,  our inequality turns into
$$
\left|\frac{z^{-\mu}}{\prod_{n=1}^\infty (1- \frac{z^2}{n^2 \pi^2})}\right| \leq \frac{x^{-\mu}}{\prod_{n=1}^\infty (1-\frac{x^2}{n^2\pi^2})}.
$$
Clearly, $|z|\geq x$ and hence $|z^{-\mu}|\leq x^{-\mu}$. To complete the proof, it suffices to verify that $|1-\frac{z^2}{n^2\pi^2}| \geq |1- \frac {x^2}{n^2\pi^2}|$ for all $n\in \N$. Write $z_n:= z/(n\pi)$ and $x_n:= \Re z_n = x/(n\pi)$. Note that $x_n\in (0,1)$.  We need to show that $|1-z_n^2| \geq |1- x_n^2|$. To this end, it suffices to check that $|1-z_n|\geq 1-x_n$ and $|1+z_n| \geq 1+x_n$. The latter inequalities are evident because $|u| \geq \Re u$ for every $u\in \C$.
\end{proof}

Having Lemma~\ref{lem:saddle_point_justify} at our disposal, we can apply the standard saddle point asymptotics~\cite[Theorem~2 on p.~423]{sidorov_fedoryuk_shabunin_book} to obtain
$$
\frac 1 {2\pi \ii} \int_{-\ii \infty + \psi(k/d)}^{\ii \infty + \psi(k/d)}  \left(f_d(z)\right)^d  \frac{\dint z}{\sin z}
\sim
\frac{(f_d(\psi(k/d)))^d}{\sqrt{2\pi d \cdot (\log f_d)''(\psi(k/d))} \cdot  \sin \psi(k/d)},
\;\;\;
d\to\infty.
$$
Note that the left-hand side is actually real because $\overline{f_d(z)} = f_d(\overline z)$, and, in fact, it is even positive by~\eqref{eq:proof_central_explicit}.  Passing to the logarithmic asymptotics and noting that $\lim_{d\to\infty} f_d(\psi(k/d)) = f_\infty(\psi(\lambda))$, we arrive at
\begin{equation}\label{eq:proof_central_2}
\lim_{d\to\infty}
\frac 1d\log \left(\frac 1 {2\pi \ii} \int_{-\ii \infty + \psi(k/d)}^{\ii \infty + \psi(k/d)}  \left(f_d(z)\right)^d  \frac{\dint z}{\sin z}\right)
=
\log f_\infty(\psi(\lambda)) = (1-\lambda) \log \psi(\lambda) - \log \sin \psi(\lambda).
\end{equation}
Inserting~\eqref{eq:proof_central_1} and~\eqref{eq:proof_central_2} into~\eqref{eq:proof_central_explicit} completes the proof of Theorem~\ref{theo:Ef_Poi_poly_linear_k}.  \hfill $\Box$

\subsection{Proof of Theorem~\ref{theo:asympt_angle}}
It is known from Theorem~2.5 in~\cite{kabluchko_poisson_zero} that
\begin{equation}\label{eq:E_alpha_explicit}
\E \alpha (C_{d+\ell, d}) = \frac{(d+\ell)!}{2\cdot  \pi^{d+\ell}} \sum_{\substack{j\in \{0,\ldots,\ell\}\\ j \text{ is even }}}
B\{d+\ell+1, d + j + 1\} (d+j)^2 A[d+j-1,-1],
\end{equation}
where the numbers  $B\{n,k\}$ with $n\in \N$ and $k\in \{1,\ldots, n\}$ are defined~\cite[Equation~(3.15)]{kabluchko_poisson_zero} as follows:
$$
B\{n,k\} = \frac{1}{(k-1)! (n-k)!} \int_0^\pi (\sin x)^{k-1} x^{n-k} \dint x.
$$
On the one hand, it follows from Theorem~\ref{theo:asympt_A_n_k} that  for every fixed $j\in \N_0$ we have
\begin{equation}\label{eq:A_n_k_asympt}
A[d+j-1,-1] \sim \frac{\sqrt 6}{d^{3/2} \sqrt \pi}, \qquad d\to\infty.
\end{equation}
On the other hand, for fixed $0\leq j \leq \ell$ we have
\begin{align}
B\{d + \ell+1, d + j+1\}
&=
\frac {1}{(d+j)!(\ell-j)!} \int_{0}^\pi (\sin x)^{d+j} x^{\ell - j} \dint x
\notag\\
&\sim
\frac {1}{(d+j)!(\ell-j)!} \cdot \left(\frac \pi 2\right)^{\ell - j} \sqrt {2\pi / d},
\qquad
d\to\infty,
\label{eq:B_n_k_asympt}
\end{align}
by the standard Laplace method~\cite[Theorem~2 on p.~405]{sidorov_fedoryuk_shabunin_book} because the maximum of $\sin x$ on $[0,\pi]$ is attained at $x= \pi/2$ and $(\sin x)''$ equals $-1$ at $x= \pi/2$. Combining~\eqref{eq:A_n_k_asympt} and~\eqref{eq:B_n_k_asympt}, we get
$$
B\{d+\ell + 1, d + j + 1\} (d+j)^2 A[d+j-1,-1] \sim \frac{2\sqrt 3}{(d+j)!}  \frac{(\pi/2)^{\ell - j}}{(\ell -j)!},
\qquad
d\to\infty.
$$
Knowing this, we see that in the sum on the right-hand side of~\eqref{eq:E_alpha_explicit} the term with $j=0$ dominates all other terms, which gives
$$
\E \alpha (C_{d+\ell, d})
\sim
\frac{(d+\ell)!}{2\cdot  \pi^{d+\ell}} \cdot \frac{2\sqrt 3}{d!}  \frac{(\pi/2)^{\ell }}{\ell!}
\sim
\sqrt 3 \,\frac{(d/2)^{\ell}}{\ell!} \pi^{-d}
,
\qquad
d\to\infty.
$$
The proof of Theorem~\ref{theo:asympt_angle} is complete.  \hfill $\Box$

\subsection{Proof of Proposition~\ref{prop:asympt_sec_moment}}\label{subsec:proof_asympt_sec_moments}
By a formula of Schneider~\cite[Theorem 1.2]{schneider_second_moments}, see~\eqref{eq:sec_moment_explicit}, we have
\begin{equation}\label{eq:schneider_second_moment_restate}
\E f_0^2(\mathcal Z_d^{\text{typ}}) = \sum_{j=0}^d a_{d,j}
\quad
\text{ with }
\quad
a_{d,j}:= \frac{2^d d! \pi^{j} }{4^j (d-j)!\Gamma(\frac j2+1)^2}.
\end{equation}
Let us characterize the asymptotic behavior of $a_{d,j}$ if $j = j(d) = [\lambda d]$ with $\lambda\in [0,1]$, as $d\to\infty$. The  Stirling formula implies that
$$
\log \Gamma(x+1) = x\log x - x + o(d)
\quad
\text{ uniformly  over }
\quad
x\in [0, d],
$$
as $d\to\infty$.
This yields
$$
\log a_{d,j} =
d\log 2 + (d\log d - d) + \lambda d \log \pi - \lambda d \log 4 -  (d-j) \log (d-j)  + (d-j) - j \log(j/2) + j + o(d),
$$
as $d\to\infty$, where the error term is uniform in $\lambda\in [0,1]$. In fact, in the above expression we can replace $j=[\lambda d]$ by $\lambda d$ resulting in an error term of at most $O(\log d)$  because
$$
\sup_{\substack{0\leq x\leq y \leq d\\ y-x\leq 1}}|(x\log x-x) - (y\log y-y)| = O(\log d), \quad
\text{ as } d\to\infty,
$$
as one easily checks using the mean value theorem and the boundedness of $x\mapsto x\log x-x$ on $[0,2]$. Replacing $j$ by $\lambda d$ and performing simple calculations  we arrive at
\begin{equation}\label{eq:sec_moment_proof_exp_profile}
\lim_{d\to\infty} \frac 1d \log a_{d, j} = \log 2 + \lambda \log \frac {\pi}2 - \lambda\log \lambda -(1-\lambda)\log (1-\lambda) = : I(\lambda)
\end{equation}
uniformly over $\lambda \in [0,1]$.  Taking the derivative, we see that the maximum of the function $I(\lambda)$ on the interval $[0,1]$ is attained at $\lambda_0 = \pi/(\pi+2)$ and the corresponding value is $I(\lambda_0) = \log (\pi +2)$.

We can now complete the proof as follows. Take some $\eps>0$. Then, for sufficiently large $d>d(\eps)$, the uniformity of convergence in~\eqref{eq:sec_moment_proof_exp_profile} implies that  $\frac 1d \log a_{d, j}\leq (1+\eps) \log (\pi +2)$ for all $j\in \{0,\ldots, d\}$.  Then, by~\eqref{eq:schneider_second_moment_restate},
$$
a_{d,[\lambda_0 d]} \leq \E f_0^2(\mathcal Z_d^{\text{typ}}) = \sum_{j=0}^d a_{d,j} \leq (d+1) \max_{j=0,\ldots, d} a_{d,j} \leq (d+1) \eee^{(1+\eps) \log (\pi +2)},
$$
provided $d>d(\eps)$ is sufficiently large.
Taking the logarithm, dividing by $d$, letting $d\to\infty$ and recalling that $\lim_{d\to\infty} \frac 1d \log a_{d,[\lambda_0 d]} = \log (\pi +2)$ by~\eqref{eq:sec_moment_proof_exp_profile}, we obtain
$$
\log (\pi + 2)
\leq
\liminf_{d\to\infty} \frac 1d \log \E f_0^2(\mathcal Z_d^{\text{typ}})
\leq
\limsup_{d\to\infty} \frac 1d \log \E f_0^2(\mathcal Z_d^{\text{typ}})
\leq
(1+\eps)\log (\pi + 2).
$$
Since $\eps>0$ is arbitrary, the proof of Proposition~\ref{prop:asympt_sec_moment} is complete.
\hfill $\Box$

\section*{Acknowledgement}
The present note has been largely  motivated by an application of Theorem~\ref{theo:asympt_A_n_k} given in~\cite{banerjee_halpern_peters}. The author expresses his gratitude to Spencer Peters and Christoph Th\"ale for stimulating discussions, to Matthias L\"owe for pointing out~\cite{fedrigo_etal}, and to an unknown referee for a careful reading of the manuscript.
Supported by the German Research Foundation under Germany's Excellence Strategy  EXC 2044 -- 390685587, \textit{Mathematics M\"unster: Dynamics - Geometry - Structure} and by the DFG priority program SPP 2265 \textit{Random Geometric Systems}.


\bibliography{poisson_polytopes_bib}
\bibliographystyle{plainnat}

\end{document}